\documentclass[12pt,twoside]{amsart}
\usepackage{amsmath, amsthm, amscd, amsfonts, amssymb, graphicx}
\usepackage[bookmarksnumbered, plainpages]{hyperref}

\newtheorem{thm}{Theorem}[section]

\newtheorem{lem}[thm]{Lemma}

\newtheorem{defn}[thm]{Definition}

\numberwithin{equation}{section}

\begin{document}

\title{\bf  Three-dimensional Lorentzian $Ein(2)$ Lie groups}
\author{Yong Wang}

\thanks{{\scriptsize
\hskip -0.4 true cm \textit{2010 Mathematics Subject Classification:}
53C40; 53C42.
\newline \textit{Key words and phrases:} Three-dimensional Lorentzian Lie groups; left-invariant metrics; $Ein(2)$ manifolds }}

\maketitle

\begin{abstract}
 In this paper, we completely classify three-dimensional Lorentzian $Ein(2)$ Lie groups.
\end{abstract}

\vskip 0.2 true cm

%------------------------------------------------------------------------------------%

\pagestyle{myheadings}
\markboth{\rightline {\scriptsize Wang}}
         {\leftline{\scriptsize Three-dimensional Lorentzian $Ein(2)$ Lie groups}}

\bigskip
\bigskip

%------------------------------------------------------------------------------------%
%------------------------------------------------------------------------------------%

\section{ Introduction}
\indent In \cite{CP}, Cordero and Parker classified three dimensional Lorentzian Lie groups equipped with a left-invariant Lorentzian metric and they wrote down the possible forms of a three-
dimensional Lie algebra, determining their curvature tensors and investigating the symmetry groups of the sectional curvature in the different cases. In \cite{Ca}, Calvaruso completely classify
three-dimensional homogenous Lorentzian manifolds, equipped with Einstein-like metrics.
In \cite{BO}, Batat and Onda studied
algebraic Ricci solitons of three-dimensional Lorentzian Lie groups. They got a complete classification of algebraic Ricci solitons of three-dimensional Lorentzian Lie groups and they proved that, contrary to the Riemannian case, Lorentzian Ricci solitons needed not be algebraic Ricci solitons.
In \cite{Be}, \cite{SK},\cite{SKAA}, the definition of $Ein(2)$ manifolds was introduced. In \cite{SAAC}, some examples of $Ein(2)$ manifolds were given. Our motivation is to give more examples of $Ein(2)$ manifolds and completely classify three-dimensional Lorentzian $Ein(2)$ Lie groups.
\\
\indent In Section 2, We classify three-dimensional $Ein(2)$ unimodular Lorentzian Lie groups.
In Section 3, We classify three-dimensional $Ein(2)$ non-unimodular Lorentzian Lie groups.

%------------------------------------------------------------------------------------%

\vskip 1 true cm

\section{ Three-dimensional $Ein(2)$ unimodular Lorentzian Lie groups}

Three-dimensional Lorentzian Lie groups had been classified in \cite{Ca1,CP}(see Theorem 2.1 and Theorem 2.2 in \cite{BO}). Throughout this paper,
we shall by $\{G_i\}_{i=1,\cdots,7}$, denote the connected, simply connected three-dimensional Lie group equipped with a left-invariant Lorentzian metric $g$ and
having Lie algebra $\{\mathfrak{g}\}_{i=1,\cdots,7}$. Let $\nabla$ be the Levi-Civita connection of $G_i$ and $R$ its curvature tensor, taken with the convention
\begin{equation}
R(X,Y)Z=\nabla_X\nabla_YZ-\nabla_Y\nabla_XZ-\nabla_{[X,Y]}Z.
\end{equation}
The Ricci tensor of $(G_i,g)$ is defined by
\begin{equation}\rho(X,Y)=-g(R(X,e_1)Y,e_1)-g(R(X,e_2)Y,e_2)+g(R(X,e_3)Y,e_3),
\end{equation}
where $\{e_1,e_2,e_3\}$ is a pseudo-orthonormal basis, with $e_3$ timelike and the Ricci operator $\rho^0$ is given by
\begin{equation}\rho(X,Y)=g({\rho^0}(X),Y).
\end{equation}
We define a symmetric $(0,2)$-tensor $\rho^2$ by
\begin{equation}
\rho^2(X,Y):=\rho(\rho^0X,Y)=g((\rho^0)^2(X),Y)=g(\rho^0(X),\rho^0(Y)).
\end{equation}
\begin{defn}\cite{Be,SK,SKAA}
$M$ is called $Ein(2)$ if $\rho^2+\lambda_1\rho+\lambda_2g=0$ for scalars $\lambda_1,\lambda_2$.
\end{defn}
So a three dimensional Lorentzian Lie group $G_i$ is $Ein(2)$ if and only if
\begin{equation}
g(\rho^0(e_i),\rho^0(e_j))+\lambda_1g(\rho^0(e_i),e_j)+\lambda_2\delta_{ij}=0,
\end{equation}
for $1\leq i\leq j\leq 3.$
By (2.1) and Lemma 3.1 in \cite{BO}, we have for $G_1$, there exists a pseudo-orthonormal basis $\{e_1,e_2,e_3\}$ with $e_3$ timelike such that the Lie
algebra of $G_1$ satisfies
\begin{equation}
[e_1,e_2]=\alpha e_1-\beta e_3,~~[e_1,e_3]=-\alpha e_1-\beta e_2,~~[e_2,e_3]=\beta e_1+\alpha e_2+\alpha e_3,~~\alpha\neq 0.
\end{equation}
\vskip 0.5 true cm
\begin{lem}(\cite{Ca},\cite{BO})
The Levi-Civita connection $\nabla$ of $G_1$ is given by
\begin{align}
&\nabla_{e_1}e_1=-\alpha e_2-\alpha e_3,~~\nabla_{e_2}e_1=\frac{\beta}{2}e_3,~~\nabla_{e_3}e_1=\frac{\beta}{2}e_2,\\\notag
&\nabla_{e_1}e_2=\alpha e_1-\frac{\beta}{2} e_3,~~\nabla_{e_2}e_2=\alpha e_3,~~\nabla_{e_3}e_2=-\frac{\beta}{2}e_1-\alpha e_3,\\\notag
&\nabla_{e_1}e_3=-\alpha e_1-\frac{\beta}{2} e_2,~~\nabla_{e_2}e_3=\frac{\beta}{2}e_1+\alpha e_2,~~\nabla_{e_3}e_3=-\alpha e_2.
\notag
\end{align}
\end{lem}
By (2.1-2.3) and Lemma 2.2, we get (see P.142 in \cite{BO})
\begin{align}
{ \rho}^0\left(\begin{array}{c}
e_1\\
e_2\\
e_3
\end{array}\right)=\left(\begin{array}{ccc}
-\frac{\beta^2}{2}&-\alpha\beta&-\alpha\beta\\
-\alpha\beta&-\left(2\alpha^2+\frac{\beta^2}{2}\right)&-2\alpha^2\\
\alpha\beta&2\alpha^2&2\alpha^2-\frac{\beta^2}{2}
\end{array}\right)\left(\begin{array}{c}
e_1\\
e_2\\
e_3
\end{array}\right).
\end{align}
We note that ${\rho}^0$ is the transport matrix of Ric in \cite{BO}.
By (2.5) and (2.8), we  have that $(G_1,g)$ is a $Ein(2)$ manifold if and only if
\begin{align}
\left\{\begin{array}{l}
\frac{\beta^4}{4}-\frac{\lambda_1\beta^2}{2}+\lambda_2=0,\\
3\alpha^2\beta^2+\frac{\beta^4}{4}-\lambda_1(2\alpha^2+\frac{\beta^2}{2})+\lambda_2=0,\\
3\alpha^2\beta^2-\frac{\beta^4}{4}+\lambda_1(-2\alpha^2+\frac{\beta^2}{2})+\lambda_2=0,\\
\alpha\beta(\beta^2-\lambda_1)=0,\\
\alpha^2(3\beta^2-2\lambda_1)=0.\\
\end{array}\right.
\end{align}
By the fifth equation in (2.9) and $\alpha\neq 0$, we get $\lambda_1=\frac{3}{2}\beta^2.$ So By the fourth equation in (2.9)
we get $\beta=0$ and $\lambda_1=0$. By the first equation in (2.9), we get $\lambda_2=0$. So we have

\vskip 0.5 true cm
\begin{thm}
$(G_1,g)$ is a $Ein(2)$ manifold if and only if $\beta=0$, $\alpha\neq 0$, $\lambda_1=\lambda_2=0$.
\end{thm}

\vskip 0.5 true cm
By (2.2) and Lemma 3.5 in \cite{BO}, we have for $G_2$, there exists a pseudo-orthonormal basis $\{e_1,e_2,e_3\}$ with $e_3$ timelike such that the Lie
algebra of $G_2$ satisfies
\begin{equation}
[e_1,e_2]=\gamma e_2-\beta e_3,~~[e_1,e_3]=-\beta e_2-\gamma e_3,~~[e_2,e_3]=\alpha e_1,~~\gamma\neq 0.
\end{equation}
\vskip 0.5 true cm
\begin{lem}(\cite{Ca},\cite{BO})
The Levi-Civita connection $\nabla$ of $G_2$ is given by
\begin{align}
&\nabla_{e_1}e_1=0,~~\nabla_{e_2}e_1=-\gamma e_2+\frac{\alpha}{2}e_3,~~\nabla_{e_3}e_1=\frac{\alpha}{2}e_2+\gamma e_3,\\\notag
&\nabla_{e_1}e_2=(\frac{\alpha}{2}-\beta) e_3,~~\nabla_{e_2}e_2=\gamma e_1,~~\nabla_{e_3}e_2=-\frac{\alpha}{2}e_1,\\\notag
&\nabla_{e_1}e_3=(\frac{\alpha}{2}-\beta) e_2,~~\nabla_{e_2}e_3=\frac{\alpha}{2}e_1,~~\nabla_{e_3}e_3=\gamma e_1.
\notag
\end{align}
\end{lem}
By (2.1-2.3) and Lemma 2.4, we get (see P.144 in \cite{BO})
\begin{align}
{ \rho}^0\left(\begin{array}{c}
e_1\\
e_2\\
e_3
\end{array}\right)=\left(\begin{array}{ccc}
-(\frac{\alpha^2}{2}+2\gamma^2)&0&0\\
0&\frac{\alpha^2}{2}-\alpha\beta&\alpha\gamma-2\beta\gamma\\
0&2\beta\gamma-\alpha\gamma&\frac{\alpha^2}{2}-\alpha\beta
\end{array}\right)\left(\begin{array}{c}
e_1\\
e_2\\
e_3
\end{array}\right).
\end{align}
By (2.5) and (2.12), we  have that $(G_2,g)$ is a $Ein(2)$ manifold if and only if
\begin{align}
\left\{\begin{array}{l}
(\frac{\alpha^2}{2}+2\gamma^2)^2-\lambda_1(\frac{\alpha^2}{2}+2\gamma^2)+\lambda_2=0,\\
(\frac{\alpha^2}{4}-\gamma^2)(\alpha-2\beta)^2+\lambda_1(\frac{\alpha^2}{2}-\alpha\beta)+\lambda_2=0,\\
(\gamma^2-\frac{\alpha^2}{4})(\alpha-2\beta)^2+\lambda_1(\alpha\beta-\frac{\alpha^2}{2})+\lambda_2=0,\\
(\alpha^2-2\alpha\beta)(2\beta\gamma-\alpha\gamma)+\lambda_1(2\beta\gamma-\alpha\gamma)=0.\\
\end{array}\right.
\end{align}
we get
\vskip 0.5 true cm
\begin{thm}
$(G_2,g)$ is a $Ein(2)$ manifold if and only if $\alpha=2\beta$, $\gamma\neq 0$, $\lambda_1=\frac{\alpha^2}{2}+2\gamma^2$, $\lambda_2=0$.
\end{thm}
\begin{proof} By the second equation and the third equation in (2.13), we get $\lambda_2=0$. By the fourth equation in (2.13) and $\gamma\neq 0$,
we get
\begin{equation}
(2\beta-\alpha)(\alpha^2-2\alpha\beta+\lambda_1)=0.
\end{equation}
\noindent Case (i) $\alpha=2\beta$. So we get $\lambda_1=\frac{\alpha^2}{2}+2\gamma^2$ by the first equation in (2.13), $\lambda_2=0$.\\
\noindent Case (ii) $\alpha\neq 2\beta$. we get $\lambda_1=2\alpha\beta-\alpha^2.$. By the second equation in (2.13), we get
$(\frac{\alpha^2}{4}-\gamma^2)(\alpha-2\beta)+\lambda_1\frac{\alpha}{2}=0.$ Then $\gamma=0$. This is a contradiction. In this case, we have no solutions. \\
\end{proof}

\vskip 0.5 true cm
By (2.3) and Lemma 3.8 in \cite{BO}, we have for $G_3$, there exists a pseudo-orthonormal basis $\{e_1,e_2,e_3\}$ with $e_3$ timelike such that the Lie
algebra of $G_3$ satisfies
\begin{equation}
[e_1,e_2]=-\gamma e_3,~~[e_1,e_3]=-\beta e_2,~~[e_2,e_3]=\alpha e_1.
\end{equation}
\vskip 0.5 true cm
\begin{lem}(\cite{Ca},\cite{BO})
The Levi-Civita connection $\nabla$ of $G_3$ is given by
\begin{align}
&\nabla_{e_1}e_1=0,~~\nabla_{e_2}e_1=a_2 e_3,~~\nabla_{e_3}e_1=a_3e_2,\\\notag
&\nabla_{e_1}e_2=a_1 e_3,~~\nabla_{e_2}e_2=0,~~\nabla_{e_3}e_2=-a_3e_1,\\\notag
&\nabla_{e_1}e_3=a_1 e_2,~~\nabla_{e_2}e_3=a_2e_1,~~\nabla_{e_3}e_3=0,
\notag
\end{align}
where
\begin{equation}
a_1=\frac{1}{2}(\alpha-\beta-\gamma),~~a_2=\frac{1}{2}(\alpha-\beta+\gamma),~~a_3=\frac{1}{2}(\alpha+\beta-\gamma).
\end{equation}
\end{lem}
By (2.1-2.3) and Lemma 2.6, we get (see P.146 in \cite{BO})
\begin{align}
{ \rho}^0=\left(\begin{array}{ccc}
-a_1a_2-a_1a_3-\beta a_2-\gamma a_3&0&0\\
0&a_2a_3-a_1a_2+\alpha a_1-\gamma a_3&0\\
0&0&-a_1a_3+a_2a_3+\alpha a_1-\beta a_2
\end{array}\right)
\end{align}
By (2.5),(2.17) and (2.18), we  have that $(G_3,g)$ is a $Ein(2)$ manifold if and only if
\begin{align}
\left\{\begin{array}{l}
\left[\frac{\alpha^2}{2}-\frac{(\beta-\gamma)^2}{2}\right]^2-\lambda_1\left[\frac{\alpha^2}{2}-\frac{(\beta-\gamma)^2}{2}\right]+\lambda_2=0,\\
\left[\frac{\beta^2}{2}-\frac{(\alpha-\gamma)^2}{2}\right]^2-\lambda_1\left[\frac{\beta^2}{2}-\frac{(\alpha-\gamma)^2}{2}\right]+\lambda_2=0,\\
\left[\frac{\gamma^2}{2}-\frac{(\alpha-\beta)^2}{2}\right]^2-\lambda_1\left[\frac{\gamma^2}{2}-\frac{(\alpha-\beta)^2}{2}\right]-\lambda_2=0.\\
\end{array}\right.
\end{align}

\vskip 0.5 true cm
\begin{thm}
$(G_3,g)$ is a $Ein(2)$ manifold if and only if \\
\noindent (i)$\alpha=\beta$. $\gamma=0$, $\lambda_2=0$.\\
\noindent (ii) $\alpha=\beta$, $\gamma\neq 0$, $\alpha\neq 0$, $\lambda_1=\frac{\gamma[(2\alpha-\gamma)^2+\gamma^2]}{4\alpha}$,
$\lambda_2=\frac{\gamma^3(-2\alpha^2+3\alpha\gamma-\gamma^2)}{4\alpha}$.\\
\noindent (iii) $\alpha=0$, $\beta\neq 0$, $\beta=\gamma$, $\lambda_2=0$.\\
\noindent (iv) $\beta=0$, $\alpha\neq 0$, $\alpha=\gamma$, $\lambda_2=0$.\\
\noindent (v) $\alpha\neq \beta$, $\alpha\beta\neq 0$, $\alpha+\beta-\gamma=0$, $\lambda_1=2\alpha\beta$, $\lambda_2=0$.\\
\noindent (vi) $\alpha\neq \beta$, $\alpha+\beta-\gamma\neq 0$, $\lambda_1=2\beta(\alpha-\beta)$, $\lambda_2=0$, $\gamma=\alpha-\beta$.\\
\noindent (vii) $\alpha\neq \beta$, $\alpha+\beta-\gamma\neq 0$, $\lambda_1=2\alpha(\beta-\alpha)$, $\lambda_2=0$, $\gamma=\beta-\alpha$.\\
\noindent (viii) $\alpha\neq \beta$, $\alpha+\beta-\gamma\neq 0$, $\gamma=\pm\sqrt{\alpha^2+\beta^2}$, $\lambda_1=\pm\sqrt{\alpha^2+\beta^2}
(\alpha+\beta\mp\sqrt{\alpha^2+\beta^2})$, $\lambda_2=\left[\frac{\alpha^2}{2}-\frac{(\beta\pm\sqrt{\alpha^2+\beta^2})^2}{2}\right]
\left[\frac{\beta^2}{2}-\frac{(\alpha\pm\sqrt{\alpha^2+\beta^2})^2}{2}\right]$.
\end{thm}
\begin{proof} The first equation minus the second equation in (2.19),
we get
\begin{equation}
(\alpha-\beta)(\alpha+\beta-\gamma)[\gamma(\alpha+\beta-\gamma)-\lambda_1]=0.
\end{equation}
\noindent Case 1) $\alpha=\beta$. By the third equation in (2.19), we get $\lambda_2=\frac{\gamma^4}{4}-\frac{\gamma^2}{2}\lambda_1.$
By the second equation in (2.19), we get
\begin{equation}
\frac{\gamma^2}{4}[(2\alpha-\gamma)^2+\gamma^2]-\lambda_1\gamma\alpha=0.
\end{equation}
\noindent Case 1)-1) $\alpha=\beta$, $\gamma=0$. So $\lambda_2=0$ and we get solution (i).\\
\noindent Case 1)-2) $\alpha=\beta$, $\gamma\neq 0$. By (2.21), we get solution (ii).\\
\noindent Case 2) $\alpha\neq \beta$, $\alpha+\beta-\gamma=0$. By the first equation in (2.19), we have $\lambda_2=0$. So by
the third equation in (2.19), we get
\begin{equation}
2\alpha^2\beta^2-\alpha\beta\lambda_1=0.
\end{equation}
\noindent Case 2)-1) $\alpha=0$, $\alpha\neq \beta$, $\alpha+\beta-\gamma=0$. We get the solution (iii).\\
\noindent Case 2)-2) $\beta=0$, $\alpha\neq0$, $\alpha\neq \beta$, $\alpha+\beta-\gamma=0$. We get the solution (iv).\\
\noindent Case 2)-3) $\alpha\beta\neq0$, $\alpha\neq \beta$, $\alpha+\beta-\gamma=0$. We get the solution (v).\\
\noindent Case 3) $\alpha\neq \beta$, $\alpha+\beta-\gamma\neq 0$. So by (2.20), we get $\lambda_1=\gamma(\alpha+\beta-\gamma)$. By the first
equation in (2.19), we get
\begin{equation}
\lambda_2=\left[\frac{\alpha^2}{2}-\frac{(\beta-\gamma)^2}{2}\right]\left[\frac{\beta^2}{2}-\frac{(\alpha-\gamma)^2}{2}\right].
\end{equation}
By the third
equation in (2.19), we get
\begin{equation}
(\beta-\alpha+\gamma)(\alpha-\beta+\gamma)(\gamma^2-\alpha^2-\beta^2)=0.
\end{equation}
\noindent Case 3)-1) $\alpha\neq \beta$, $\alpha+\beta-\gamma\neq 0$, $\beta-\alpha+\gamma=0$. We get the solution (vi).\\
\noindent Case 3)-2) $\alpha\neq \beta$, $\alpha+\beta-\gamma\neq 0$, $\alpha-\beta+\gamma=0$. We get the solution (vii).\\
\noindent Case 3)-3) $\alpha\neq \beta$, $\alpha+\beta-\gamma\neq 0$, $\gamma^2-\alpha^2-\beta^2=0$. We get the solution (viii).\\
\end{proof}
\vskip 0.5 true cm
By (2.4) and Lemma 3.11 in \cite{BO}, we have for $G_4$, there exists a pseudo-orthonormal basis $\{e_1,e_2,e_3\}$ with $e_3$ timelike such that the Lie
algebra of $G_4$ satisfies
\begin{align}
[e_1,e_2]=-e_2+(2\eta-\beta)e_3,~~\eta=1~{\rm or}-1,~~[e_1,e_3]=-\beta e_2+ e_3,~~[e_2,e_3]=\alpha e_1.
\end{align}
\vskip 0.5 true cm
\begin{lem}(\cite{Ca},\cite{BO})
The Levi-Civita connection $\nabla$ of $G_4$ is given by
\begin{align}
&\nabla_{e_1}e_1=0,~~\nabla_{e_2}e_1=e_2+b_2e_3,~~\nabla_{e_3}e_1=b_3e_2-e_3,\\\notag
&\nabla_{e_1}e_2=b_1 e_3,~~\nabla_{e_2}e_2=- e_1,~~\nabla_{e_3}e_2=-b_3e_1,\\\notag
&\nabla_{e_1}e_3=b_1 e_2,~~\nabla_{e_2}e_3=b_2e_1,~~\nabla_{e_3}e_3=- e_1,
\notag
\end{align}
where
\begin{equation}
b_1=\frac{\alpha}{2}+\eta-\beta,~~b_2=\frac{\alpha}{2}-\eta,~~b_3=\frac{\alpha}{2}+\eta.
\end{equation}
\end{lem}
\vskip 0.5 true cm
By (2.1-2.3) and Lemma 2.8, we get (see P.147 in \cite{BO})
\begin{align}
{ \rho}^0=\left(\begin{array}{ccc}
-\frac{\alpha^2}{2}&0&0\\
0&\frac{\alpha^2}{2}+2\eta(\alpha-\beta)-\alpha\beta+2&-\alpha+2\beta-2\eta\\
0&\alpha-2\beta+2\eta&\frac{\alpha^2}{2}-\alpha\beta-2+2\eta\beta
\end{array}\right)
\end{align}
By (2.5) and (2.28), we  have that $(G_4,g)$ is a $Ein(2)$ manifold if and only if
\begin{align}
\left\{\begin{array}{l}
\frac{\alpha^4}{4}-\lambda_1\frac{\alpha^2}{2}+\lambda_2=0,\\
\left[\frac{\alpha^2}{2}+2\eta(\alpha-\beta)-\alpha\beta+2\right]^2-
(\alpha-2\beta+2\eta)^2+
\left[\frac{\alpha^2}{2}+2\eta(\alpha-\beta)-\alpha\beta+2\right]\lambda_1+\lambda_2=0,\\
(\frac{\alpha^2}{2}-\alpha\beta-2+2\eta\beta)^2-(\alpha-2\beta+2\eta)^2+(\frac{\alpha^2}{2}-\alpha\beta-2+2\eta\beta)\lambda_1-\lambda_2=0,\\
(\alpha-2\beta+2\eta)[\alpha(\alpha-2\beta+2\eta)+\lambda_1]=0.\\
\end{array}\right.
\end{align}
\begin{thm}
$(G_4,g)$ is a $Ein(2)$ manifold if and only if \\
\noindent (i) $\alpha=0$, $\beta=\eta$, $\lambda_2=0$.\\
\noindent (ii) $\alpha\neq 0$, $\beta=\frac{\alpha}{2}+\eta$, $\lambda_2=0$, $\lambda_1=\frac{\alpha^2}{2}$.\\
\noindent (iii) $\alpha=0$, $\beta\neq \eta$, $\lambda_1=0$, $\lambda_2=0$.
\end{thm}
\begin{proof} \noindent Case 1) $\alpha-2\beta+2\eta=0$. By (2.29), we have $\lambda_2=0$ and $\alpha^4=2\lambda_1\alpha^2$.\\
\noindent Case 1)-1)$\alpha-2\beta+2\eta=0$, $\alpha=0$. We get the solution (i).\\
\noindent Case 1)-2)$\alpha-2\beta+2\eta=0$, $\alpha\neq 0$. We get the solution (ii).\\
\noindent Case 2) $\alpha-2\beta+2\eta\neq 0$. So $\lambda_1=-\alpha(\alpha-2\beta+2\eta).$ By the first equation in (2.29),
we get $\lambda_2=-\frac{\alpha^3}{2}(\frac{3\alpha}{2}-2\beta+2\eta).$ Using the expression of $\lambda_1$ and the second equation minusing the third equation in (2.29), we get $\lambda_2=0$. So $\alpha^3(\frac{3\alpha}{2}-2\beta+2\eta)=0$.\\
\noindent Case 2)-1) $\alpha-2\beta+2\eta\neq 0$, $\alpha=0$. We get the solution (iii).\\
\noindent Case 2)-2) $\alpha-2\beta+2\eta\neq 0$, $\alpha\neq 0$. Then $\frac{3\alpha}{2}-2\beta+2\eta=0$ and $\lambda_1=\frac{\alpha^2}{2}$
and $\beta=\frac{3}{4}\alpha+\eta.$ By the second equation in (2.29), we get $\alpha=0$. This is a contradiction and we have no solutions in this case.\\
\end{proof}
\section{ Three-dimensional $Ein(2)$ non-unimodular Lorentzian Lie groups.}

\indent By (2.5) and Lemma 4.1 in \cite{BO}, we have for $G_5$, there exists a pseudo-orthonormal basis $\{e_1,e_2,e_3\}$ with $e_3$ timelike such that the Lie
algebra of $G_5$ satisfies
\begin{equation}
[e_1,e_2]=0,~~[e_1,e_3]=\alpha e_1+\beta e_2,~~[e_2,e_3]=\gamma e_1+\delta e_2,~~\alpha+\delta\neq 0,~~\alpha\gamma+\beta\delta=0.
\end{equation}
\vskip 0.5 true cm
\begin{lem}(\cite{Ca},\cite{BO})
The Levi-Civita connection $\nabla$ of $G_5$ is given by
\begin{align}
&\nabla_{e_1}e_1=\alpha e_3,~~\nabla_{e_2}e_1=\frac{\beta+\gamma}{2}e_3,~~\nabla_{e_3}e_1=-\frac{\beta-\gamma}{2}e_2,\\\notag
&\nabla_{e_1}e_2=\frac{\beta+\gamma}{2} e_3,~~\nabla_{e_2}e_2=\delta e_3,~~\nabla_{e_3}e_2=\frac{\beta-\gamma}{2}e_1,\\\notag
&\nabla_{e_1}e_3=\alpha e_1+\frac{\beta+\gamma}{2}  e_2,~~\nabla_{e_2}e_3=\frac{\beta+\gamma}{2} e_1+\delta e_2,~~\nabla_{e_3}e_3=0.
\notag
\end{align}
\end{lem}
By (2.1-2.3) and Lemma 3.1, we get (see P.149 in \cite{BO})
\begin{align}
{ \rho}^0=\left(\begin{array}{ccc}
\alpha^2+\alpha\delta+\frac{\beta^2-\gamma^2}{2}&0&0\\
0&\alpha\delta+\delta^2-\frac{\beta^2-\gamma^2}{2}&0\\
0&0&\alpha^2+\delta^2+\frac{(\beta+\gamma)^2}{2}.
\end{array}\right)
\end{align}
By (2.5) and (3.3), we  have that $(G_5,g)$ is a $Ein(2)$ manifold if and only if
\begin{align}
\left\{\begin{array}{l}
(\alpha^2+\alpha\delta+\frac{\beta^2-\gamma^2}{2})^2+\lambda_1(\alpha^2+\alpha\delta+\frac{\beta^2-\gamma^2}{2})+\lambda_2=0,\\
(\alpha\delta+\delta^2-\frac{\beta^2-\gamma^2}{2})^2+\lambda_1(\alpha\delta+\delta^2-\frac{\beta^2-\gamma^2}{2})+\lambda_2=0,\\
(\alpha^2+\delta^2+\frac{(\beta+\gamma)^2}{2})^2+\lambda_1(\alpha^2+\delta^2+\frac{(\beta+\gamma)^2}{2})-\lambda_2=0.\\
\end{array}\right.
\end{align}
\begin{thm}
$(G_5,g)$ is a $Ein(2)$ manifold if and only if \\
\noindent (i) $\gamma=-\beta$, $\alpha=\delta$, $\delta\neq 0$, $\lambda_1=-2\alpha^2$, $\lambda_2=0$.\\
\noindent (ii) $\alpha=\beta=\gamma=0$, $\delta\neq 0$, $\lambda_1=-\delta^2$, $\lambda_2=0$.\\
\noindent (iii) $\alpha\neq0$, $\beta=\gamma=\delta=0$, $\lambda_1=-\alpha^2$, $\lambda_2=0$.\\
\noindent (iv) $\alpha^2+\beta^2-\delta^2-\gamma^2\neq 0$, $\alpha+\delta\neq 0$, $\beta\neq 0$, $\delta=-\frac{\alpha\gamma}{\beta}$,
$\lambda_1=-(\alpha+\delta)^2$,
\begin{equation}
\lambda_2=\alpha\delta(\alpha+\delta)^2+\frac{(\beta^2-\gamma^2)(\delta^2-\alpha^2)}{2}-\frac{(\beta^2-\gamma^2)^2}{4}.
\end{equation}
\begin{equation}
\alpha^2=\frac{-\frac{\beta}{2}[(\beta^2-\gamma^2)^2+(\beta+\gamma)^4]\pm\sqrt{\bigtriangleup}}{2\gamma(3\beta^2+3\gamma^2-2\gamma\beta)},
\end{equation}
where
\begin{equation}
\bigtriangleup=\frac{\beta^2}{4}[(\beta^2-\gamma^2)^2+(\beta+\gamma)^4]^2-\beta^3\gamma(3\beta^2+3\gamma^2-2\gamma\beta)
[(\beta^2-\gamma^2)^2+(\beta+\gamma)^4].
\end{equation}
\end{thm}
\begin{proof} The first equation minus the second equation in (3.4), then we get
\begin{equation}
(\alpha^2+\beta^2-\delta^2-\gamma^2)[(\alpha+\delta)^2+\lambda_1]=0.
\end{equation}
\noindent Case 1) $\alpha^2+\beta^2-\delta^2-\gamma^2=0$. By $\alpha+\delta\neq 0$, the second equation plusing the third equation in (3.4), we get
\begin{equation}
\lambda_1=-\frac{(\alpha\delta+\delta^2-\frac{\beta^2-\gamma^2}{2})^2+(\alpha^2+\delta^2+\frac{(\beta+\gamma)^2}{2})^2}{\alpha^2+\alpha\delta
+2\delta^2+\gamma^2+\beta\gamma}.
\end{equation}
By (3.9) and the third equation in (3.4), we get
\begin{equation}
\lambda_2=\frac{[\alpha^2+\delta^2+\frac{(\beta+\gamma)^2}{2}](\alpha\delta+\delta^2-\frac{\beta^2-\gamma^2}{2})(\alpha^2-\alpha\delta+\beta\gamma
+\beta^2)}
{\alpha^2+\alpha\delta
+2\delta^2+\gamma^2+\beta\gamma}.
\end{equation}
\noindent Case 1)-1) $\alpha^2+\beta^2-\delta^2-\gamma^2=0$, $\beta\neq 0$. Then $\delta=-\frac{\alpha\gamma}{\beta}$ and $\beta^2=\gamma^2$ and $\alpha^2=\delta^2.$\\
\noindent Case 1)-1)-1) $\alpha^2+\beta^2-\delta^2-\gamma^2=0$, $\beta\neq 0$, $\beta=\gamma$. We get $\alpha+\delta=0$ and this is a contradiction. In this case, we have no solutions.\\
\noindent Case 1)-1)-2) $\alpha^2+\beta^2-\delta^2-\gamma^2=0$, $\beta\neq 0$, $\beta=-\gamma$. We get $\alpha=\delta$ and $\alpha\neq 0$.
By (3.9) and (3.10), we get the solution (i).\\
Case 1)-2) $\alpha^2+\beta^2-\delta^2-\gamma^2=0$, $\beta= 0$. So $\alpha\gamma=0.$\\
Case 1)-2)-1) $\alpha^2+\beta^2-\delta^2-\gamma^2=0$, $\beta= 0$, $\alpha=0$. So $-\delta^2-\gamma^2=0$ and $\delta=0$ and $\alpha+\delta=0$.
This is a contradiction. In this case, we have no solutions.\\
Case 1)-2)-2) $\alpha^2+\beta^2-\delta^2-\gamma^2=0$, $\beta= 0$, $\gamma=0$. We have $\alpha=\delta\neq 0$ and by (3.9) and (3.10), we get the
solution (i).\\
\noindent Case 2) $\alpha^2+\beta^2-\delta^2-\gamma^2\neq 0$. By (3.8), we get $\lambda_1=-(\alpha+\delta)^2$. So by the second equation and
third equation in (3.4), we have (3.5) and
\begin{equation}
(\alpha+\delta)^2[\alpha^2+\alpha\delta+\frac{\beta^2-\gamma^2}{2}+\alpha^2+\delta^2+\frac{(\beta+\gamma)^2}{2}]
=(\alpha^2+\delta^2+\frac{(\beta+\gamma)^2}{2})^2+(\alpha^2+\alpha\delta+\frac{\beta^2-\gamma^2}{2})^2.
\end{equation}
\noindent Case 2)-1) $\alpha^2+\beta^2-\delta^2-\gamma^2\neq 0$, $\beta\neq 0$. So $\delta=-\frac{\alpha\gamma}{\beta}$ and by (3.11), we get
\begin{equation}
\gamma(3\beta^2+3\gamma^2-2\beta\gamma)\alpha^4+\frac{\beta}{2}[(\beta^2-\gamma^2)^2+(\beta+\gamma)^4]\alpha^2
+\frac{\beta^3}{4}[(\beta^2-\gamma^2)^2+(\beta+\gamma)^4]=0.
\end{equation}
By (3.12), we get (3.6) and the solution (iv).\\
\noindent Case 2)-2) $\alpha^2+\beta^2-\delta^2-\gamma^2\neq 0$, $\beta= 0$. So $\alpha\gamma=0$.\\
\noindent Case 2)-2)-1) $\alpha^2+\beta^2-\delta^2-\gamma^2\neq 0$, $\alpha=\beta= 0$. By (3.11), we get the solution (ii).\\
\noindent Case 2)-2)-2) $\alpha^2+\beta^2-\delta^2-\gamma^2\neq 0$, $\gamma=\beta= 0$. By (3.11), we get $\delta=0$ and the solution (iii).\\
\end{proof}
\noindent {\bf Remark.} We give an example of (iv) in Theorem 3.2 now. Let $\beta=-1$ and $\gamma=2$, then $\delta=2\alpha$ and $\alpha+\delta
=3\alpha\neq 0$. $\alpha^2+\beta^2-\delta^2-\gamma^2=\frac{1}{\beta^2}(\alpha^2+\beta^2)(\beta^2-\gamma^2)\neq 0$. (3.12) becomes
$76\alpha^4-10\alpha^2-5=0$, so $\alpha^2=\frac{5+\sqrt{405}}{76}.$ By (3.5), we have $\lambda_1=-\frac{45+9\sqrt{405}}{76}$,
$\lambda_2=18\alpha^4-\frac{9}{2}\alpha^2-\frac{9}{4}.$\\
\vskip 0.5 true cm
By (2.6) and Lemma 4.3 in \cite{BO}, we have for $G_6$, there exists a pseudo-orthonormal basis $\{e_1,e_2,e_3\}$ with $e_3$ timelike such that the Lie
algebra of $G_6$ satisfies
\begin{equation}
[e_1,e_2]=\alpha e_2+\beta e_3,~~[e_1,e_3]=\gamma e_2+\delta e_3,~~[e_2,e_3]=0,~~\alpha+\delta\neq 0£¬~~\alpha\gamma-\beta\delta=0.
\end{equation}
\vskip 0.5 true cm
\begin{lem}(\cite{Ca},\cite{BO})
The Levi-Civita connection $\nabla$ of $G_6$ is given by
\begin{align}
&\nabla_{e_1}e_1=0,~~\nabla_{e_2}e_1=-\alpha e_2-\frac{\beta-\gamma}{2}e_3,~~\nabla_{e_3}e_1=\frac{\beta-\gamma}{2}e_2-\delta e_3,\\\notag
&\nabla_{e_1}e_2=\frac{\beta+\gamma}{2} e_3,~~\nabla_{e_2}e_2=\alpha e_1,~~\nabla_{e_3}e_2=-\frac{\beta-\gamma}{2}e_1,\\\notag
&\nabla_{e_1}e_3=\frac{\beta+\gamma}{2}  e_2,~~\nabla_{e_2}e_3=-\frac{\beta-\gamma}{2} e_1,~~\nabla_{e_3}e_3=-\delta e_1.
\notag
\end{align}
\end{lem}
\vskip 0.5 true cm
By (2.1-2.3) and Lemma 3.3, we get (see P.150 in \cite{BO})
\begin{align}
{ \rho}^0=\left(\begin{array}{ccc}
-\alpha^2-\delta^2+\frac{(\beta-\gamma)^2}{2}&0&0\\
0&-\alpha^2-\alpha\delta+\frac{\beta^2-\gamma^2}{2}&0\\
0&0&-\delta^2-\alpha\delta-\frac{\beta^2-\gamma^2}{2}.
\end{array}\right)
\end{align}
By (2.5) and (3.15), we  have that $(G_6,g)$ is a $Ein(2)$ manifold if and only if
\begin{align}
\left\{\begin{array}{l}
\left[\alpha^2+\delta^2-\frac{(\beta-\gamma)^2}{2}\right]^2-\lambda_1[\alpha^2+\delta^2-\frac{(\beta-\gamma)^2}{2}]+\lambda_2=0,\\
\left[\alpha^2+\alpha\delta-\frac{\beta^2-\gamma^2}{2}\right]^2-\lambda_1[\alpha^2+\alpha\delta-\frac{\beta^2-\gamma^2}{2}]+\lambda_2=0,\\
-(\delta^2+\alpha\delta+\frac{\beta^2-\gamma^2}{2})^2+\lambda_1(\delta^2+\alpha\delta+\frac{\beta^2-\gamma^2}{2})+\lambda_2=0.\\
\end{array}\right.
\end{align}

\begin{thm}
$(G_6,g)$ is a $Ein(2)$ manifold if and only if \\
\noindent (i) $\gamma=\beta\neq 0$, $\alpha=\delta\neq 0$, $\lambda_1=2\alpha^2$, $\lambda_2=0$.\\
\noindent (ii) $\beta=\gamma=\delta=0$, $\alpha\neq 0$, $\lambda_1=\alpha^2$, $\lambda_2=0$.\\
\noindent (iii) $\gamma=\beta= 0$, $\alpha=\delta\neq 0$, $\lambda_1=2\alpha^2$, $\lambda_2=0$.\\
\noindent (iv) $\beta\neq \gamma$, $\delta=\gamma\neq 0$, $\alpha=\beta$, $\alpha+\delta\neq 0$, $\lambda_1=\frac{(\alpha+\delta)^2}{2}$,
$\lambda_2=0$.\\
\noindent (v) $\beta\neq \gamma$, $\delta=\gamma=0$, $\alpha\neq 0$, $\lambda_1=\frac{\alpha^4-\alpha^2\beta^2+\frac{\beta^4}{2}}{\alpha^2}$,
$\lambda_2=\frac{\beta^2(\alpha^2-\frac{\beta^2}{2})(\beta^2-\alpha^2)}{2\alpha^2}$.\\
\noindent (vi) $\beta\neq \gamma$, $\delta\neq \gamma$, $\delta=-\gamma\neq 0$, $\alpha=-\beta$, $\alpha+\delta\neq 0$, $\lambda_1=\frac{(\alpha+\delta)^2}{2}$,
$\lambda_2=0$.\\
\noindent (vii) $\beta\neq 0$, $\delta=\frac{\alpha\gamma}{\beta}$, $\delta^2-\alpha\delta+\beta\gamma-\gamma^2\neq 0$, $\alpha+\delta\neq 0$,
$\lambda_1=2\alpha^2+\delta^2+\alpha\delta+\beta\gamma-\beta^2$,
\begin{equation}
\lambda_2=(2\alpha^2+\delta^2+\alpha\delta+\beta\gamma-\beta^2)\left[\alpha^2+\delta^2-\frac{(\beta-\gamma)^2}{2}\right]-
\left[\alpha^2+\delta^2-\frac{(\beta-\gamma)^2}{2}\right]^2,
\end{equation}
\begin{equation}
\alpha^2=\frac{\gamma\beta(\beta-\gamma)\pm\sqrt{\gamma^2\beta^2(\gamma-\beta)^2+2\beta^3(\beta+\gamma)(\beta-\gamma)^2}}{2(\beta+\gamma)}.
\end{equation}
\noindent (viii) $\alpha=\beta=\gamma=0$, $\delta\neq 0$, $\lambda_1=\delta^2$, $\lambda_2=0$.\\
\noindent (viiii) $\alpha=\beta=0$, $\gamma\neq 0$, $\delta^2=\frac{\gamma^2}{2}$, $\lambda_1=\delta^2$, $\lambda_2=0$.\\
\end{thm}
\begin{proof}
The first equation minusing the second equation in (3.16), then we get
\begin{equation}
(\delta^2-\alpha\delta+\beta\gamma-\gamma^2)(2\alpha^2+\delta^2+\alpha\delta+\beta\gamma-\beta^2-\lambda_1)=0.
\end{equation}
\noindent Case 1) $\delta^2-\alpha\delta+\beta\gamma-\gamma^2=0$. The second equation minusing the third equation in (3.16), then we get
\begin{equation}
\lambda_1=\frac{\left[\alpha^2+\alpha\delta-\frac{\beta^2-\gamma^2}{2}\right]^2+(\delta^2+\alpha\delta+\frac{\beta^2-\gamma^2}{2})^2}
{(\alpha+\delta)^2}.
\end{equation}
By (3.20) and the third equation in (3.16), we get
\begin{equation}
\lambda_2=\frac{(\alpha^2+\alpha\delta-\frac{\beta^2-\gamma^2}{2})(\delta^2+\alpha\delta+\frac{\beta^2-\gamma^2}{2})(\delta^2-\alpha^2+
\beta^2-\gamma^2)}
{(\alpha+\delta)^2}.
\end{equation}
By $\delta^2-\alpha\delta+\beta\gamma-\gamma^2=0$ and $\alpha\gamma=\beta\delta$, we get $(\beta-\gamma)(\delta^2-\gamma^2)=0$.\\
\noindent Case 1)-1) $\delta^2-\alpha\delta+\beta\gamma-\gamma^2=0$, $\beta=\gamma\neq 0$. Then $\alpha=\delta\neq 0$. By (3.20) and (3.21),
we get the solution (i).\\
\noindent Case 1)-2) $\delta^2-\alpha\delta+\beta\gamma-\gamma^2=0$, $\beta=\gamma= 0$. Then $\delta^2=\alpha\delta$.\\
\noindent Case 1)-2)-1) $\delta^2-\alpha\delta+\beta\gamma-\gamma^2=0$, $\beta=\gamma= 0$, $\delta=0$, By (3.20) and (3.21),
we get the solution (ii).\\
\noindent Case 1)-2)-2) $\delta^2-\alpha\delta+\beta\gamma-\gamma^2=0$, $\beta=\gamma= 0$, $\delta\neq 0$. Then $\alpha=\delta$. By (3.20) and (3.21), we get the solution (iii).\\
\noindent Case 1)-3) $\delta^2-\alpha\delta+\beta\gamma-\gamma^2=0$, $\beta\neq\gamma$. So $\delta^2=\gamma^2$. We let $\delta=\gamma\neq 0$.
By (3.20) and (3.21),
we get the solution (iv).\\
\noindent Case 1)-4) $\delta^2-\alpha\delta+\beta\gamma-\gamma^2=0$, $\beta\neq\gamma$, $\delta=\gamma=0$. By (3.20) and (3.21),
we get the solution (v).\\
\noindent Case 1)-5) $\delta^2-\alpha\delta+\beta\gamma-\gamma^2=0$, $\beta\neq\gamma$, $\delta=-\gamma$. By (3.20) and (3.21),
we get the solution (vi).\\
\noindent Case 2) $\delta^2-\alpha\delta+\beta\gamma-\gamma^2\neq 0$. By (3.19), we have
\begin{equation}
\lambda_1=2\alpha^2+\delta^2+\alpha\delta+\beta\gamma-\beta^2.
\end{equation}
By the first equation in (3.16), we get
\begin{equation}
\lambda_2=(2\alpha^2+\delta^2+\alpha\delta+\beta\gamma-\beta^2)\left[\alpha^2+\delta^2-\frac{(\beta-\gamma)^2}{2}\right]-
\left[\alpha^2+\delta^2-\frac{(\beta-\gamma)^2}{2}\right]^2.
\end{equation}
By the first equation and the third equation in (3.16), we get
\begin{equation}
(\alpha^2+\alpha\delta+\frac{\gamma^2-\beta^2}{2})
\left[\alpha^2+\delta^2-\frac{(\beta-\gamma)^2}{2}\right]=
(\delta^2+\alpha\delta+\frac{-\gamma^2+\beta^2}{2})[-2\alpha^2+\beta^2-\gamma^2+\frac{(\beta-\gamma)^2}{2}].
\end{equation}
\noindent Case 2)-1) $\delta^2-\alpha\delta+\beta\gamma-\gamma^2\neq 0$, $\beta\neq 0$. Then $\delta=\frac{\alpha\gamma}{\beta}$. By (3.24), we get
\begin{equation}
(\beta+\gamma)\alpha^4+\gamma\beta(\gamma-\beta)\alpha^2-\frac{\beta^3}{2}(\beta-\gamma)^2=0.
\end{equation}
So we get (3.18) and the solution (vii).\\
\noindent Case 2)-2) $\delta^2-\alpha\delta+\beta\gamma-\gamma^2\neq 0$, $\beta=0$. Then $\alpha\gamma=0$\\
\noindent Case 2)-2)-1) $\delta^2-\alpha\delta+\beta\gamma-\gamma^2\neq 0$, $\beta=0$, $\alpha=0$. By (3.24), we get $\gamma^2(\delta^2-\frac{\gamma^2}{2})=0$.\\
\noindent Case 2)-2)-1)-1) $\delta^2-\alpha\delta+\beta\gamma-\gamma^2\neq 0$, $\beta=0$, $\alpha=0$, $\gamma=0$. By (3.23), we get the solution (viii).\\
\noindent Case 2)-2)-1)-2) $\delta^2-\alpha\delta+\beta\gamma-\gamma^2\neq 0$, $\beta=0$, $\alpha=0$, $\gamma\neq 0$. So $\delta^2=\frac{\gamma^2}{2}$. By (3.22) and (3.23), we get the solution (viiii).\\
\noindent Case 2)-2)-2) $\delta^2-\alpha\delta+\beta\gamma-\gamma^2\neq 0$, $\beta=0$, $\alpha\neq 0$. So $\gamma=0$. By (3.24), we get
$\alpha=0$ and this is a contradiction. In this case, we have no solutions.
\end{proof}

\noindent {\bf Remark.} We give an example of (vii) in Theorem 3.4 now. Let $\beta=1$ and $\gamma=2$, then $\delta=2\alpha$ and
$\alpha^2=\frac{-2+\sqrt{10}}{6}$ and $\delta^2-\alpha\delta+\beta\gamma-\gamma^2=2\alpha^2-2\neq 0$. $\alpha+\delta=3\alpha\neq 0$,
$\lambda_1=\frac{4\sqrt{10}-5}{3}$, $\lambda_2=\frac{37-8\sqrt{10}}{12}$.\\

By (2.7) and Lemma 4.5 in \cite{BO}, we have for $G_7$, there exists a pseudo-orthonormal basis $\{e_1,e_2,e_3\}$ with $e_3$ timelike such that the Lie
algebra of $G_7$ satisfies
\begin{equation}
[e_1,e_2]=-\alpha e_1-\beta e_2-\beta e_3,~~[e_1,e_3]=\alpha e_1+\beta e_2+\beta e_3,~~[e_2,e_3]=\gamma e_1+\delta e_2+\delta e_3,,~~\alpha+\delta\neq 0,~~\alpha\gamma=0.
\end{equation}
\vskip 0.5 true cm
\begin{lem}(\cite{Ca},\cite{BO})
The Levi-Civita connection $\nabla$ of $G_7$ is given by
\begin{align}
&\nabla_{e_1}e_1=\alpha e_2+\alpha e_3,~~\nabla_{e_2}e_1=\beta e_2+(\beta+\frac{\gamma}{2})e_3,~~\nabla_{e_3}e_1=-(\beta-\frac{\gamma}{2})e_2-\beta e_3,\\\notag
&\nabla_{e_1}e_2=-\alpha e_1+\frac{\gamma}{2} e_3,~~\nabla_{e_2}e_2=-\beta e_1+\delta e_3,~~\nabla_{e_3}e_2=(\beta-\frac{\gamma}{2})e_1-\delta e_3,\\\notag
&\nabla_{e_1}e_3=\alpha e_1+\frac{\gamma}{2} e_2,~~\nabla_{e_2}e_3=(\beta+\frac{\gamma}{2})e_1+\delta e_2,~~\nabla_{e_3}e_3=-\beta e_1-\delta e_2.
\notag
\end{align}
\end{lem}

\vskip 0.5 true cm
By (2.1-2.3) and Lemma 3.5, we get (see P.151 in \cite{BO})
\begin{align}
{ \rho}^0=\left(\begin{array}{ccc}
-\frac{\gamma^2}{2}&0&0\\
0&\alpha\delta-\alpha^2-\beta\gamma+\frac{\gamma^2}{2}&-\alpha^2+\alpha\delta-\beta\gamma\\
0&\alpha^2-\alpha\delta+\beta\gamma&-\alpha\delta+\alpha^2+\beta\gamma+\frac{\gamma^2}{2}.
\end{array}\right)
\end{align}
By (2.5) and (3.28), we  have that $(G_7,g)$ is a $Ein(2)$ manifold if and only if
\begin{align}
\left\{\begin{array}{l}
\frac{\gamma^4}{4}-\lambda_1\frac{\gamma^2}{2}+\lambda_2=0,\\
(\alpha\delta-\alpha^2-\beta\gamma+\frac{\gamma^2}{2})^2-(\alpha^2-\alpha\delta+\beta\gamma)^2+
(\alpha\delta-\alpha^2-\beta\gamma+\frac{\gamma^2}{2})\lambda_1+\lambda_2=0,\\
(-\alpha\delta+\alpha^2+\beta\gamma+\frac{\gamma^2}{2})^2-(\alpha^2-\alpha\delta+\beta\gamma)^2+
(-\alpha\delta+\alpha^2+\beta\gamma+\frac{\gamma^2}{2})\lambda_1-\lambda_2=0,\\
(\alpha^2-\alpha\delta+\beta\gamma)(\lambda_1+\gamma^2)=0.\\
\end{array}\right.
\end{align}
\begin{thm}
$(G_7,g)$ is a $Ein(2)$ manifold if and only if \\
\noindent (i) $\alpha=\beta=\gamma=0$, $\delta\neq 0$, $\lambda_2=0$.\\
\noindent (ii) $\alpha=\gamma=0$, $\beta\neq 0$, $\delta\neq 0$, $\lambda_2=0$\\
\noindent (iii) $\alpha\neq 0$, $\gamma=0$, $\alpha=\delta$, $\lambda_2=0$\\
\noindent (iv) $\alpha\neq 0$, $\gamma=0$, $\alpha\neq\pm\delta$, $\lambda_1=\lambda_2=0$.\\
\end{thm}
\begin{proof} We know $\alpha\gamma=0$.\\
\noindent Case 1) $\alpha=0$. By the fourth equation in (3.29), we have $\beta\gamma(\lambda_1+\gamma^2)=0.$\\
\noindent Case 1)-1) $\alpha=\beta=0$. By the second equation and the third equation in (3.29), we get $\lambda_2=0$.
By the first equation and the second equation in (3.29), we get $\gamma=0$. We get the solution (i).\\
\noindent Case 1)-2) $\alpha=0$, $\beta\neq 0$. Then $\gamma(\lambda_1+\gamma^2)=0$.\\
\noindent Case 1)-2)-1) $\alpha=0$, $\beta\neq 0$, $\gamma=0$. Then we get the solution (ii).\\
\noindent Case 1)-2)-2) $\alpha=0$, $\beta\neq 0$, $\gamma\neq 0$. Then $\lambda_1=-\gamma^2$ and $\lambda_2=-\frac{3}{4}\gamma^4$
by the first equation in (3.29). By the second equation in (3.29), we get $\gamma=0$. This is a contradiction.\\
\noindent Case 2) $\alpha\neq 0$. So $\gamma=0$. By the first equation in (3.29), we get $\lambda_2=0$. By the fourth equation in (3.29), we have $(\alpha-\delta)\lambda_1=0$.\\
\noindent Case 2)-1) $\alpha\neq 0$, $\gamma=0$, $\lambda_2=0$, $\alpha=\delta$. We get the solution (iii).\\
\noindent Case 2)-2) $\alpha\neq 0$, $\gamma=0$, $\lambda_2=0$, $\alpha\neq\delta$. Then $\lambda_1=0$ and we get the solution (iv).\\
\end{proof}

\vskip 1 true cm

\section{Acknowledgements}

The author was supported in part by NSFC No.11771070.

\vskip 1 true cm

%-----------------------------------------------------------------------------
%-----------------------------------------------------------------------------

\bigskip
\bigskip

\noindent {\footnotesize {\it Y. Wang} \\
{School of Mathematics and Statistics, Northeast Normal University, Changchun 130024, China}\\
{Email: wangy581@nenu.edu.cn}

\end{document}